\documentclass[review]{elsarticle}
\usepackage{lineno}
\usepackage[latin1]{inputenc}
\usepackage{times}
\usepackage{amssymb,mathrsfs, amsmath}
\usepackage{amsmath,amsthm}
\usepackage{amssymb}
\usepackage{latexsym}
\usepackage{amsfonts}
\usepackage{mathrsfs}
\usepackage{epsfig}
\usepackage{hyperref}
\usepackage{graphicx}
\usepackage{graphics}

\newtheorem{thm}{Theorem}[section]

\newtheorem{defn}{Definition}[section]
\newtheorem{prop}{Proposition}[section]

\journal{...}
\begin{document}

\begin{frontmatter}

\title{Cohomology and Deformations of n-Hom-Liebniz algebra morphisms}

%%% Group authors per affiliation:
%%\author{\fnref{myfootnote}}
%%\address{}
%%\fntext[myfootnote]{This author is supported by CSIR,\textsc{India}.}
\author{RB Yadav \fnref{myfootnote}\corref{mycorrespondingauthor}}
\address{Sikkim University, Gangtok, Sikkim, 737102, \textsc{India}}
\cortext[mycorrespondingauthor]{Corresponding author}
\ead{rbyadav15@gmail.com}
%\author{Goutam Mukherjee\fnref{myfootnote}\corref{mycorrespondingauthor}}
%\address{Stat-Math Division,  Indian Statistical Institute, Kolkata, 700108, \textsc{India}}
%\ead{goutam@isical.ac.in}
%\fntext[myfootnote]{This author is supported by CSIR,\textsc{India}.}
%
%%
%\author{RB Yadav}
%\address{Indian Statistical Institute Tezpur, Assam 784028, \textsc{India}}
%\ead{rbyadav15@gmail.com}
\begin{abstract}
In this paper, we  introduce cohomology of n-Hom-Liebniz algebra morphisms and formal deformation theory of n-Hom-Liebniz algebra morphisms .

\end{abstract}

\begin{keyword}
\texttt{Hochschild cohomology, formal deformations, modules}
\MSC[2010] 13D03 \sep 13D10\sep 14D15 \sep 	16E40
\end{keyword}

\end{frontmatter}
%\linenumbers

%\maketitle

\section{Introduction}\label{rbsec1}
  M. Gerstenhaber introduced algebraic deformation theory  in \cite{MG1},\cite{MG2},\cite{MG3}, \cite{MG4}, \cite{MG5}. He studied deformation theory of associative algebras.  Deformation theory of associative algebra morphisms was studied by M. Gerstenhaber and S.D. Schack \cite{GS1}, \cite{GS2}, \cite{GS3}. Deformation theory of Lie algebras was studied by Nijenhuis and Richardson \cite{NR1}, \cite{NR2}. Algebraic deformations of modules were first studied  by Donald and Flanigan \cite{DF}. Deformation theory of Leibniz algebra morphisms was studied in \cite{AM}. Deformation theory of module homomorphisms was studied in \cite{RLYY}. Cohomology and deformations of n-Hom-Liebniz algebras were introduced in \cite{NM}.\\
 Above works inspire us to work on cohomology and deformations of n-Hom-Liebniz algebra morphisms.

Organization of the paper is as follows. In Section \ref{rbsec2}, we recall some definitions and results.  In Section \ref{rbsec3}, we introduce  deformation complex and  deformation cohomology of an n-Hom-Liebniz algebra homomorphism. In Section \ref{rbsec4}, we introduce  deformation of  an n-Hom-Liebniz algebra  homomorphism. In Section \ref{rbsec5}, we introduce equivalence of  deformations of  an n-Hom-Liebniz algebra  homomorphism.

% In this section  we prove one of our most important  results that obstructions to deformations are cocycles. In Section \ref{rbsec5}, we study equivalence of two deformations and rigidity of a module  homomorphism.

\section{Preliminaries}\label{rbsec2}
From \cite{NM}, we recall following definitions.
\begin{defn}
  A representation of a multiplicative n-Hom-Leibniz algebra $(L, [.,\cdots,.], \alpha)$
is a pair $(M, \alpha_M)$, where M is a vector space and $\alpha_M : M \to M$ is a linear map, equipped
with n actions
$[\cdots,\cdots,\cdots]_i : L^{\otimes i}\otimes M \otimes L^{\otimes n-1-i} \to M, \;\;0 \le i \le n-1$,
satisfying $(2n-1)$ equations which are obtained from
\begin{equation}\label{}
  [[x_1,\cdots,x_n],\alpha( y_1),\cdots,\alpha( y_{n-1})]=\sum_{i=1}^{n}[\alpha(x_1),\cdots, [x_i, y_1,\cdots, y_{n-1}],\alpha(x_{i+1}),\cdots,\alpha(x_n)]
\end{equation}
by letting exactly one of the variables $x_1, \cdots, x_n, y_1,\cdots, y_{n-1}$ be in M (hence the corresponding $\alpha$ should be replaced by $\alpha_M$ ) and all others in L.
\end{defn}

\begin{defn}\label{rbd2}
  Let $(M, \alpha_M)$ be a representation of the multiplicative n-Hom-Leibniz algebra $(L, [.,\cdots,.], \alpha)$. We define  $C^p(L,M)$  as the collection of  linear maps $f : L\otimes (D^{n-1}(L))^{\otimes p-1} \to M$ such that $\alpha_M \circ f = f \circ (\alpha \otimes \bar{\alpha}^ {\otimes p-1}$). We define $$\delta_p : C^p(L,M) \to  C^{p+1}(L,M) $$ by
\begin{align}\label{}
   &\delta^p(f)(z,X_1,\cdots ,X_p) \nonumber\\
   &=\sum_{1\le i<j}^{p}(-1)^jf(\alpha(z),\bar{\alpha}(X_1),\cdots,\bar{\alpha}(X_{i-1}),[X_i,X_j],\bar{\alpha}(X_{i+1}),
   \cdots,\hat{X_j},\cdots,\bar{\alpha}(X_p)) \nonumber\\
   & + \sum_{i=1}^{p}(-1)^if([z,X_i], \bar{\alpha}(X_1),\cdots,\hat{X_i},\cdots,\bar{\alpha}(X_p) \nonumber\\
   & + \sum_{i=1}^{p}(-1)^{i+1}[f(z,X_1,\cdots,\hat{X_i},\cdots,X_p),\bar{\alpha}^{\otimes p-1}(X_i)]_0 \nonumber\\
   & +(X_{1.\alpha}f( ,X_2,\cdots ,X_p)).\alpha^{p-1}(z),
\end{align}
where
\begin{eqnarray*}
   && (X_{1.\alpha}f( ,X_2,\cdots ,X_p)).\alpha^{p-1}(z) \nonumber \\
   &&= \sum_{i=1}^{n-1}[\alpha^{p-1}(z),\alpha^{p-1}(X_1^1),\cdots, f(X_1^i,X_2,\cdots,X_p),\cdots), \alpha^{p-1}(X_1^p)]_i\nonumber
\end{eqnarray*}
The last two terms in the definition of the coboundary make use of the n actions
$[\cdots,\cdots,\cdots]_i : L^{\otimes i}\otimes M \otimes L^{\otimes n-1-i} \to M, \;\;0 \le i \le n-1$.
Elements of $C^p(L,M)$ are called p-cochains.
\end{defn}

\begin{defn}
  A formal one-parameter deformation of a n-Hom-Leibniz algebra  $(L, [.,\cdots,.], \alpha)$ is a map $f_t : L[[t]]^{\otimes n}\to  L[[t]]$,
where tensor product taken over $K[[t]]$, such that for $ X = (x_1, x_2,\cdots , x_n) \in L^{\otimes n}$, $f_t(X) =\sum_{i=0}^{\infty}t^iF_i(X)$
for some $F_i :L^{\otimes n}\to L$, $i \ge 1$, $F_0$ being the bracket in L, such that
% and $f_t$ defining a n-Hom-Leibniz algebra structure on L[[t]].Explicitly, this would mean
\begin{eqnarray}\label{rbd3}
   && f_t(f_t(X), \alpha(y_1),\cdots , \alpha(y_{n-1})) \nonumber \\
   && =\sum_{i=1}^{n} f_t(\alpha(x_1),\cdots , \alpha(x_{i-1}), f_t(x_i, y_1,\cdots, y_{n-1}), \alpha(x_{i+1}),\cdots, \alpha(x_n))
\end{eqnarray}

 \ref{rbd3} is equivalent to
 \begin{eqnarray}\label{rbd4}
  &&\sum_{i+j=l}  F_i(F_j(X), \alpha(y_1),\cdots , \alpha(y_{n-1}))\nonumber \\
  &&= \sum_{i=1}^{n}\sum_{j+k=l} F_j(\alpha(x_1),\cdots , \alpha(x_{i-1}), F_k(x_i, y_1,\cdots, y_{n-1}), \alpha(x_{i+1}),\cdots, \alpha(x_n)),\nonumber\\
    &&
 \end{eqnarray}
for all  integers $0\le l$.
\end{defn}
From \ref{rbd4}, we have
\begin{eqnarray}
   &&  [(F_l(X), \alpha(y_1),\cdots , \alpha(y_{n-1})]+ F_l([X], \alpha(y_1),\cdots , \alpha(y_{n-1}))\nonumber \\
   &&- \sum_{i=1}^{n}[\alpha(x_1),\cdots , \alpha(x_{i-1}), F_l(x_i, y_1,\cdots, y_{n-1}), \alpha(x_{i+1}),\cdots, \alpha(x_n)]\nonumber\\
   &&-\sum_{i=1}^{n} F_l(\alpha(x_1),\cdots , \alpha(x_{i-1}), [x_i, y_1,\cdots, y_{n-1}], \alpha(x_{i+1}),\cdots, \alpha(x_n))\nonumber\\
   &&= \sum_{i=1}^{n}\sum_{\substack{j+k=l\\j,k>0}} F_j(\alpha(x_1),\cdots , \alpha(x_{i-1}), F_k(x_i, y_1,\cdots, y_{n-1}), \alpha(x_{i+1}),\cdots, \alpha(x_n))\nonumber \\
   &&-\sum_{\substack{i+j=l\\ i,j>0}}  F_i(F_j(X), \alpha(y_1),\cdots , \alpha(y_{n-1}))\nonumber  \\
   &&
\end{eqnarray}

\section{ Deformation complex of n-Hom-Liebniz algebra morphism}\label{rbsec3}
%In this setion, we introduce  deformation complex of a module  homorphism.  In the subsequent sections we show that second and third cohomology of this complex controls deformation.

\begin{defn}
Let  $(L, [.,\cdots,.], \alpha)$, $(M, [.,\cdots,.], \beta)$ be n-Hom-Liebniz algebras and $\phi:M\to N$ be a  n-Hom-Liebniz algebra morphism. Let $(C^{\ast}(L;L),\delta)$,  $(C^{\ast}(M;M),\delta)$ and  $(C^{\ast}(L;M),\delta)$ be  as defined in \ref{rbd2} with coeffiecients in $L$, $M$ and $M$  respectively.
We define
$$C^p(\phi)=C^p(L;L)\oplus C^p(M;M)\oplus C^{p-1}(L;M),$$ for all $p\in \mathbb{N}$ and $C^0(\phi)=0$.
For any n-Hom-Liebniz algebra morphism $\phi:M\to N $, $u\in C^p(L;L)$, $v\in C^p(M;M)$, define    $\phi u:L\otimes (D^{n-1}(L))^{\otimes p-1} \to M$ and $v\phi:L\otimes (D^{n-1}(L))^{\otimes p-1} \to M$ by $\phi u(z,X_1, X_2,\cdots, X_{p-1})=\phi( u(z,X_1, X_2,\cdots, X_{p-1}))$,   $v\phi(z,X_1, X_2,\cdots, X_{p-1})=v(\phi z, \phi^{\otimes p-1}X_1, \phi^{\otimes p-1}X_2, \cdots,  \phi^{\otimes p-1}X_{p-1}),$ for all $(z, x_1, X_2,\cdots, X_{p-1})\in L\otimes (D^{n-1}(L))^{\otimes p-1} .$
Also, we define $d^p:C^p(\phi)\to C^{p+1}(\phi)$ by $$d^p(u,v,w)=(\delta^p u, \delta^p v, \phi u-v\phi -\delta^{p-1}w ),$$ for all $(u,v,w)\in C^p(\phi).$ Here the $\delta^p$'s denote coboundaries of the cochain complexes $C^n(L;L)$,  $C^n(M;M)$ and $ C^{p-1}(L;M)$.
\end{defn}
\begin{prop}
  $(C^{\ast}(\phi), d)$ is a cochain complex.
\end{prop}
\begin{proof}
  We have \begin{eqnarray*}
           d^{p+1}d^p(u,v,w)   &=& d^{p+1}(\delta^p u, \delta^p v, \phi u-v\phi -\delta^{p-1}w ) \\
             &=& (\delta^{p+1}\delta^p u, \delta^{p+1}\delta^p v, \phi(\delta^p u)- (\delta^p v)\phi- \delta^p( \phi u-v\phi -\delta^{p-1}w ))
          \end{eqnarray*} One can easily see that $\delta^p( \phi u-v\phi)=\phi (\delta^p u)- (\delta^p v)\phi$. So, since $\delta^{p+1}\delta^p u=0,$ $\delta^{p+1}\delta^p v=0$, $\delta^{p+1}\delta^p w=0$, we have $d^{p+1}d^n=0$. Hence we conclude the result.
\end{proof}
We call the cochain complex $(C^{\ast}(\phi),d)$  as deformation complex of $\phi,$ and the corresponding cohomology as  deformation cohomology of $\phi$. We denote the  deformation cohomology by $H^p(\phi)$, that is $H^p(\phi)=H^p(C^{\ast}(\phi),d)$. Next proposition relates $H^{\ast}(\phi)$ to $H^{\ast}(L,L)$,  $H^{\ast}(M,M)$ and $H^{\ast}(L,M)$.

\begin{prop}\label{rb-99}
  If $H^n(L,L)=0$,  $H^n(M,M)=0$ and $H^{n-1}(L,M)=0$, then $H^{n}(\phi)=0$.
\end{prop}
\begin{proof}
  Let $(u,v,w)\in C^n(\phi)$ be a cocycle, that is $d^n(u,v,w)=(\delta^n u, \delta^n v, \phi u-v\phi -\delta^{n-1}w )=0$. This implies that $\delta^n u=0$, $\delta^n v=0$, $\phi u-v\phi -\delta^{n-1}w =0$. $H^n(L,L)=0 \Rightarrow u=\delta^{n-1}u_1$ and  $H^n(M,M)=0\Rightarrow \delta^{n-1}v_1=v$, for some $u_1\in C^{n-1}(L,L)$ and $v_1\in C^{n-1}(M,M)$. So $0=\phi u-v\phi -\delta^{n-1}w=\phi(\delta^{n-1}u_1)-(\delta^{n-1}v_1)\phi-\delta^{n-1}w=\delta^{n-1}(\phi u_1)-\delta^{n-1}(v_1\phi)-\delta^{n-1}w=\delta^{n-1}(\phi u_1-v_1\phi-w)$. So $\phi u_1-v_1\phi-w\in C^{n-1}(L,M)$ is a cocycle. Now,  $H^{n-1}(L,M)=0\Rightarrow \phi u_1-v_1\phi-w=\delta^{n-2}w_1$, for some $w_1\in C^{n-2}(L,M)\Rightarrow \phi u_1-v_1\phi-\delta^{n-2}w_1=w$. Thus $(u,v,w)=(\delta^{n-1} u_1, \delta^{n-1 }v_1, \phi u_1-v_1\phi -\delta^{n-2}w_1 )=d^{n-1}(u_1,v_1,w_1)$, for some $(u_1,v_1,w_1)\in C^{n-1}(\phi).$  Thus every cocycle in $C^n(\phi)$ is a coboundary. Hence we conclude that $H^{n}(\phi)=0$.
\end{proof}
\section{ Cohomology of  a n-Hom-Liebniz algebra morphisms}\label{rbsec4}
In this section we introduce a cochain complex for every n-Hom-Liebniz algebra morphisms and call it deformation complex of the n-Hom-Liebniz algebra morphism. We   define Cohomology  of  a n-Hom-Liebniz algebra morphism  as the cohomologyof this cochain complex.
 
\begin{defn}\label{rb2}
Let  $(L, [.,\cdots,.], \alpha)$ and  $(M, [.,\cdots,.], \beta)$ be n-Hom-Liebniz algebras. A formal one-parameter deformation of a n-Hom-Liebniz algebra  morphism $\phi:L\to M$ is a triple $(\xi_t, \eta_t, \phi_t)$, in which:

\begin{enumerate}
  \item  $\xi_t=\sum_{i=0}^{\infty}\xi_i t^i$ is a formal one-parameter deformation for $L$.
  \item  $\eta_t=\sum_{i=0}^{\infty}\eta_it^i$ is a formal one-parameter deformation for  $M$.
  \item  $\phi_t=\sum_{i=0}^{\infty}\phi_it^i$, where $\phi_i:L\to M$ is  a module  homomorphism, $1\le i$, such that $\phi_t(\xi_t(X))=\eta_t\phi_t^{\otimes n}(X),$  for all $X\in L^{\otimes n}$ and $\phi_0=\phi$.
\end{enumerate}

Therefore a triple  $(\xi_t, \eta_t, \phi_t)$, as given above, is a formal one-parameter deformation of $\phi$ provided following properties are satisfied.
\begin{itemize}
  \item[(i)] \begin{eqnarray}\label{rbd20}
   && \xi_t(\xi_t(X), \bar{\alpha}(Y)) \nonumber \\
   && =\sum_{i=1}^{n} \xi_t(\alpha(x_1),\cdots , \alpha(x_{i-1}), \xi_t(x_i, y_1,\cdots, y_{n-1}), \alpha(x_{i+1}),\cdots, \alpha(x_n)),\nonumber\\
   &&
\end{eqnarray}
for all $X\in L^{\otimes n},\;\; Y\in L^{\otimes n-1},\;\; X=(x_1,\cdots,x_n),\;\; Y=(y_1,\cdots y_{n-1});$
   \item[(ii)] \begin{eqnarray}\label{rbd21}
   && \eta_t(\eta_t(X), \bar{\beta}(Y)) \nonumber \\
   && =\sum_{i=1}^{n} \eta_t(\beta(x_1),\cdots , \beta(x_{i-1}), \eta_t(x_i, y_1,\cdots, y_{n-1}), \beta(x_{i+1}),\cdots,\beta(x_n)),\nonumber\\
   &&
\end{eqnarray}
   for $X\in M^{\otimes n},\;\; Y\in M^{\otimes n-1},\;\; X=(x_1,\cdots,x_n),\;\; Y=(y_1,\cdots y_{n-1});$
  \item[(iii)]
  \begin{equation}\label{rbd22}
    \phi_t(\xi_t X)=\eta_t\phi_t^{\otimes n}(X),
  \end{equation}
  for all $X\in L^{\otimes n}$.
\end{itemize}
The conditions \ref{rbd20}, \ref{rbd21} and \ref{rbd22} are equivalent to following conditions respectively.

 \begin{eqnarray}\label{rbeqn1}
  &&\sum_{i+j=l}  \xi_i(\xi_j(X),\bar{\alpha}(Y))\nonumber \\
  &&= \sum_{i=1}^{n}\sum_{j+k=l} \xi_j(\alpha(x_1),\cdots , \alpha(x_{i-1}), \xi_k(x_i, y_1,\cdots, y_{n-1}), \alpha(x_{i+1}),\cdots, \alpha(x_n)),\nonumber\\
    &&
 \end{eqnarray}
for  $X\in L^{\otimes n},\;\; Y\in L^{\otimes n-1},\;\; X=(x_1,\cdots,x_n),\;\; Y=(y_1,\cdots y_{n-1})$ and   integers $0\le l$.
    \begin{eqnarray}\label{rbeqn2}
  &&\sum_{i+j=l}  \eta_i(\eta_j(X), \bar{\beta}Y)\nonumber \\
  &&= \sum_{i=1}^{n}\sum_{j+k=l} \eta_j(\beta(x_1),\cdots , \eta(x_{i-1}), \eta_k(x_i, y_1,\cdots, y_{n-1}), \beta(x_{i+1}),\cdots, \beta(x_n)),\nonumber\\
    &&
 \end{eqnarray}
for  $X\in M^{\otimes n},\;\; Y\in M^{\otimes n-1},\;\; X=(x_1,\cdots,x_n),\;\; Y=(y_1,\cdots y_{n-1})$ and   integers $0\le l$.
      \begin{eqnarray}\label{rbeqn3}
         &&\sum_{i+j=l}\phi_i\xi_j(X) \nonumber \\
         &&=\sum_{\substack{\{i+j_1+\cdots +j_n\}=l}}\eta_i(\phi_{j_1}x_1,\cdots,\phi_{j_n}x_n)
      \end{eqnarray}

\end{defn}

\section{Equivalence of  deformations}\label{rbsec5}

 Recall from \cite{DY1} that a formal isomorphism between the  deformations $\xi_t$ and $\tilde{\xi_t}$ of a module  M is a $k[[t]]$-linear automorphism  $\Psi_t:M[[t]]\to M[[t]]$ of the  form  $\Psi_t=\sum_{i\ge 0}\psi_it^i$, where each $\psi_i$ is a $k$-linear map $M\to M$, $\psi_0(a)=a$, for all $a\in A$ and $\tilde{\xi_t}(r)\Psi_t(m)=\Psi_t(\xi_t(r)m),$ for all $r\in A$, $m\in M$
\begin{defn}
 Let   $(\xi_t,\eta_t,\phi_t)$  and $(\tilde{\xi_t},\tilde{\eta_t},\tilde{\phi_t})$ be two  deformations of $\phi$.  A  formal isomorphism from $(\xi_t,\eta_t,\phi_t)$  to  $(\tilde{\xi_t},\tilde{\eta_t},\tilde{\phi_t})$ is a pair  $(\Psi_t,\Theta_t)$, where $\Psi_t:M[[t]]\to M[[t]]$ and $\Theta_t:N[[t]]\to N[[t]]$ are formal isomorphisms from $\xi_t$ to $\tilde{\xi_t}$ and $\eta_t$ to $\tilde{\eta_t}$, respectively,  such that $$\tilde{\phi_t}\circ\Psi_t=\Theta_t\circ\phi_t.$$
  Two formal  deformations $(\xi_t,\eta_t,\phi_t)$  and $(\tilde{\xi_t},\tilde{\eta_t},\tilde{\phi_t})$ are said to be equivalent if there exists a formal isomorphism  $(\Psi_t,\Theta_t)$ from $(\xi_t,\eta_t,\phi_t)$ to  $(\tilde{\xi_t},\tilde{\eta_t},\tilde{\phi_t})$.
\end{defn}
\begin{defn}
  Any  deformation of $\phi:M\to N$ that is equialent to the deformation $(\xi_0,\eta_0,\phi)$ is said to be a trivial deformation.
\end{defn}
\begin{thm}
  The cohomology class of the infinitesimal of a  deformation $(\xi_t,\eta_t,\phi_t)$ of $\phi:A\to B$ is determined by the equivalence class of $(\xi_t,\eta_t,\phi_t)$.
\end{thm}
\begin{proof}
  Let  $(\Psi_t,\Theta_t)$ from  $(\xi_t,\eta_t,\phi_t)$ to  $(\tilde{\xi_t},\tilde{\eta_t},\tilde{\phi_t})$ be a  formal  isomorphism. So, we have  $\tilde{\xi_t}\Psi_t=\Psi_t\xi_t,$ $\tilde{\eta_t}\Theta_t=\Theta_t \eta_t,$ and   $\tilde{\phi_t}\circ\Psi_t=\Theta_t\circ\phi_t.$ This implies that $\xi_1-\tilde{\xi_1}=\delta^0\psi_1$, $\eta_1-\tilde{\eta_1}=\delta^0\theta_1$ and $\phi_1-\tilde{\phi_1}=\phi\psi_1-\theta_1\phi$. So we have $d^1(\psi_1,\theta_1,0)=(\xi_1,\eta_1,\phi_1)-(\tilde{\xi_1},\tilde{\eta_1},\tilde{\phi_1}).$ This finishes the proof.
\end{proof}

\end{document}